\documentclass[11pt,reqno]{article}
\input{epsf.sty}
\usepackage{hyperref}
\usepackage[final]{epsfig}
\usepackage{color}
\usepackage{amsmath, amsthm, amssymb, dsfont, mathrsfs}
\newtheorem {thm}{Theorem}[section]
\newtheorem {prop}[thm]{Proposition}

\newtheorem {cor}[thm]{Corollary}
\newtheorem {defn}[thm]{Definition}

\def\N{{\Bbb N}}
\def\TT{{\Bbb T}}
\def\Z{{\Bbb Z}}

\def\P{{\Bbb P}}

\def\one{{\mathds 1}}
\def\Gxi{{\GG(\g[\xi])}}
\def\Gexxi{{{\rm ex}\GG(\g[\xi])}}
\newcommand{\F}{{\mathcal F}}

\newcommand{\C}{{\mathcal C}}

\def\0{{\bf 0}}

\def\ba{{\backslash}}

\def\a{\alpha}

\def\b{\beta}

\def\d{\delta}

\def\phi{\varphi}

\def\g{\gamma}

\def\l{\lambda}

\def\k{\kappa}

\def\s{\sigma}

\def\x{\xi}

\def\o{\omega}

\def\D{\Delta}

\def\L{\Lambda}

\def\O{\Omega}

\def\T{\T}

\def\C{{\cal C}}
\def\DD{{\cal D}}
\def\GG{{\cal G}}
\def\PP{{\cal P}}

\def\FF{{\cal F}}
\def\HH{{\cal H}}
\def\TT{{\cal T}}
\def\EE{{\cal E}}

\begin{document}

\title{Extremal decomposition for random Gibbs measures: From general metastates to metastates on extremal random Gibbs measures} 

\author{
Codina Cotar\footnote{
Department of Statistical Science, University College London, 
1-19 Torrington Place, London WC1E 6BT, UK
\texttt{c.cotar@ucl.ac.uk}}
,
Benedikt Jahnel\footnote{
Weierstrass Institute for Applied Analysis and Stochastics, Mohrenstrasse 39, 10117 Berlin, Germany
\texttt{Benedikt.Jahnel@wias-berlin.de}}
\, and
Christof K\"ulske\footnote{
Ruhr-University of Bochum, Fakult\"at f\"ur Mathematik, Postfach 102148, 44721, Bochum,
Germany
\texttt{Christof.Kuelske@ruhr-uni-bochum.de}}
}

\newcommand{\CC}[1]{{\color{blue} #1}}
\newcommand{\DE}[1]{{\color{red} #1}}
\newcommand{\CK}[1]{{\color{green} #1}}

\maketitle

\begin{abstract}
The concept of metastate measures on the states of a random spin system was introduced to be able to treat the large-volume asymptotics for complex quenched random systems, like spin glasses, which may exhibit chaotic volume dependence in the strong-coupling regime. We consider the general issue of the extremal decomposition for Gibbsian specifications which depend measurably on a parameter that may describe a whole random environment in the infinite volume. Given a random Gibbs measure, as a measurable map from the environment space, we prove measurability of its decomposition measure on pure states at fixed environment, with respect to the environment. As a general corollary we obtain that, for any metastate, there is an associated 
decomposition metastate, which is supported on the extremes for almost all environments, and which has the same barycenter. 
\end{abstract}

\smallskip
\noindent {\bf AMS 2010 subject classification:} 82B44, 60K35
\bigskip 

{\em Keywords: Gibbs measures, disordered systems, extremal decomposition, metastates} 

\section{Introduction} 
In this note, we consider quenched random spin models in the infinite volume on lattices or more generally countably-infinite graphs. This includes examples like the random-field Ising model and the Edwards-Anderson nearest-neighbor spin glass, but also more generally continuous spin models. We prove an extended version of the known extremal decomposition (see~\cite{Fo75} and~\cite[Theorem 7.26]{Ge11}) for the infinite-volume Gibbs measures in terms of pure states, i.e., the extremal Gibbs measures, which is {\em measurable} w.r.t.~the random environment, see Theorem~\ref{theorem1}.
We also present a connection between this result and the theory of metastates via the notion of the {\em decomposition metastate}, see Corollary~\ref{theorem3}. 

What are the difficulties? Indeed, measurability w.r.t.~random environments in the infinite volume, which can be taken for granted in simple situations, becomes nontrivial for general systems. The difficulty comes from the following physical phenomenon that may happen for general random systems and in particular those with frustrated interactions like spin glasses. 
A system is frustrated, if it has competing interactions, in the sense that it is impossible to 
minimize the total finite-volume Hamiltonian 
by minimizing all individual local interactions terms in the same configuration. 
Hence a groundstate has to be determined as a result of a non-trivial optimization which 
may depend on the precise shape and size of the volume. The same holds for low-temperature states.  
There may be {\em chaotic volume-dependence}, see~\cite{NeSt97,NeSt98,NeSt13}, of the finite-volume Gibbs measures along increasing volume sequences with fixed nonrandom boundary conditions, in the phase-transition regime. Hence, to construct a measurable random Gibbs measure, i.e., a {\em measurable map} $\xi\mapsto\mu[\xi]$ from the environment to the Gibbs measures via $\xi$-dependent volume 
sequences, is in general difficult.
Conversely, for fixed increasing volume sequence, pure states may only be found by $\xi$-dependent boundary conditions. This phenomenon of possible chaotic volume-dependence leads to the natural definition of a {\em metastate} $\k[\xi](d \mu)$ as a probability measure on the (random) Gibbs measures of the system, see~\cite{NeSt97,Bo06}. From a physical point of view, this is not just an unwanted complication, but it is an important object for the description of the large-volume asymptotics. Its intuitive meaning is that it describes the limiting empirical measure for the occurrence of Gibbs states for fixed environment $\xi$ in a sufficiently sparsely chosen increasing volume sequence, see equation~\eqref{Empi}. Hence, it carries the additional information about the weight, or relevance, of a particular Gibbs measure, compare~\cite{IaKu10}. To show existence of a metastate, e.g., in situations with a compact local spin space, is always possible, see~\cite{AiWe90,NeSt97,Bo06} for Ising systems, and as we sketch in Section~\ref{Metastates}. Note that metastates were used to show uniqueness results for short-range spin-glass models with independent couplings in two dimensions, see~\cite{ArDaNeSt10}, in higher dimensions, see~\cite{ArNeSt18}, and for other applications, 
see also~\cite{Bo06}. Having constructed the metastate, we note that there is always an associated measurable map from the environment to the Gibbs measures, which appears 
as its barycenter of mean value over all the Gibbs measures the metastate sees.

In the present note, we therefore start with the assumption of existence of such a measurable Gibbs measure and investigate its extremal decomposition. In general, there is no reason to assume that a metastate is supported only on extremal Gibbs measures. For a (mean-field) example where this occurs see, e.g., the Hopfield model in~\cite{Ku97}. Thus, we need to discuss the decomposition into pure states also at the metastate level.

We keep the setup as general as possible and assume only measurability of the Gibbsian specification for most of our results. We do not need to make the requirements of quasilocality or non-nullness, which are otherwise very important in the theory of infinite-volume Gibbs measures, see~\cite{Ge11}.

Our results may be useful for random systems, but  they may be applied also to other classes of parametrized specifications, be it by finite-dimensional or infinite-dimension parameters.

\section{The extreme decomposition}
\subsection{Setup, examples, and first theorem}\label{Sec_Setup}
Let $S$ be a countable set of sites and $(E,\EE)$ the local state space of each site. We assume that $(E,\EE)$ is a standard Borel space, see~\cite{Ge11} for details. We consider models on $\O=E^S$ equipped with the product sigma-algebra $\FF=\mathcal E^S$, given by \emph{specifications} $\g[\xi]$ which depend on a random \emph{parameter} $\xi\in H$, where $(H, \HH,\P)$ is a probability space.
A (non-random) specification $\g=(\g_\L)_{\L\Subset S}$ is a family of probability kernels from $\O$ to the set of probability measures $\PP(\O)=\PP(\O,\FF)$ indexed by finite sets of sites $\L\Subset S$.  We want to impose only minimal conditions on the specification. First, we require consistency, i.e., for all finite volumes $\D\subset\L$, events $A\in\FF$ and boundary conditions $\o\in\O$ we have that $\int\g_\L(d\s|\o)\g_\D(A|\s)=\g_\L(A|\o)$. Second, we require properness, i.e., for all $\L\Subset S$ we have that $\g_\L(A|\o)=\one_A(\o)$ if $A$ is measurable outside of $\L$. We do not require for example non-nullness (lower bounds on the specification) or quasilocality (continuity of the specification w.r.t.~the boundary condition), which play an important role for the existence theory and variational characterization of Gibbs measures. 
In what follows, we will assume that the random specification $\g[\cdot]$ is measurable in the sense that for all $\L\Subset S$ and $A\in \FF$, the mapping 
$$(\xi,\o)\mapsto \g_\L[\xi](A|\o)$$
is $(\HH\otimes\F)-\mathcal{B}[0,1]$ measurable. 

\medskip
We think of the parameter $\xi$ mostly, but not exclusively, as describing disorder.  
Often we imagine a Gibbsian specification given in terms of an absolutely 
summable potential $\Phi[\xi]=(\Phi_A[\xi])_{A\Subset S}$ where each of the 
potential functions $\o\mapsto\Phi_A[\xi](\o)$ depends on the spin variable $\o$ only inside $A$, and 
depends measurably on the parameter $\xi$ and the configuration $\o$. Then, we may build the associated quenched specification as the $\l$-modification, in which the product measure $\l^\L$ is modified by an exponential factor. 
Keeping the 
parameter $\xi$ fixed, this means that we define $\xi$-dependent 
specification kernels via
\begin{equation}
\g_{\L}^{\Phi}[\xi](d\o_\L|\eta)=\frac{\l^{\L}(d\o_{\L})\exp(-\sum_{A\cap\L\neq \emptyset}\Phi_A[\xi](\o_{\L}\eta_{\L^c}) )
}{
\int_{E^\L}\l^{\L}(d\tilde\o_{\L})\exp(-\sum_{A\cap\L\neq \emptyset}\Phi_A[\xi](\tilde\o_{\L}\eta_{\L^c}) )
},
\end{equation}
where $\lambda \in \PP(E)$ is some a priori measure on the single site space $E$ and $\l^\L$ denotes the $|\L|$-fold product distribution on $E^\L$. We write $\o_\L\eta_{\L^c}$ for the concatenation of the configuration $\o_\L$ in $\L$ and $\eta_{\L^c}$ in $\L^c=S\setminus\L$.
Two examples which fit this framework  are the \emph{random-field Ising model}, and the \emph{Edwards-Anderson nearest-neighbor spin glass}. We may consider these models on a general graph with countable vertex set $S$, with specification kernels defined as modifications w.r.t.~the uniform measure $\l$ on $E=\{-1,1\}$.

For the random-field Ising model, the potential functions are given by the pair interaction terms $\Phi_{A}[\xi](\o)=-\b\o_i\o_j$ for $A=\{i,j\}$ for nearest neighbors $i,j\in S$, and single-site terms $\Phi_{A}[\xi](\o)=\xi_i\o_i$ for $A=\{i\}$. For all other $A\Subset S$ we set $\Phi_{A}[\xi]=0$. Note that the disorder configuration $\xi=(\x_i)_{i\in S}$ enters only via the single-site terms, and are usually assumed to be given by i.i.d.~random variables w.r.t.~$\P$ on $(H,\HH)$.

For the Edwards-Anderson spin glass (in zero external field), only the nearest-neighbor pair-interaction term is non-vanishing with $\Phi_{A}[\xi](\o)=\b\x_{i,j}\o_i\o_j$ for $A=\{i,j\}$ and $i,j$ nearest neighbors, where $\xi$ is an $H$-valued configuration on the edges $\{i, j\}$ of the underlying graph. One usually assumes that $\x_{i,j}$ are mean-zero random variables, which are i.i.d.~over edges.

Both models are different in the following sense. For the random-field Ising model, monotonicity arguments (based on the FKG inequality) guarantee existence of a maximal and a minimal limiting measure, for any disorder configuration, which simplifies the analysis substantially. Nevertheless, to decide in which dimensions, and when they are equal, is far from trivial, see~\cite{BrKu88,AiWe90,Bo06,AiPe18}. On the other hand, for free or periodic boundary conditions, the random-field Ising model in three dimensions does show chaotic size-dependence.

For the Edwards-Anderson spin glass, the aforementioned chaotic volume-dependence is expected to occur for deterministic boundary conditions. On the other hand, if there is only one pair of Gibbs states, related by spin-flip symmetry (which should happen in two dimensions), any volume sequence with symmetry-preserving boundary conditions cannot cause chaotic volume-dependence. Examples of such boundary conditions are free and periodic boundary conditions. There may be chaotic volume-dependence also for these boundary conditions in higher dimensions, if there should be many of such pairs of Gibbs measures.
In that sense, both models display ordinary and size-dependent convergence, at least in some regimes and for complementary sets of boundary conditions. 
As another similarity, in the two examples above, disorder enters via local terms in a Hamiltonian which is used to define the Gibbsian specification, and the disorder takes values in a product space. This is nice, but our setup is more general and assumes only measurability of the specification depending on 
a parameter $\xi$.

A lack of monotonicity occurs also for gradient Gibbs measures. These appear for unbounded spin systems where one is forced to restrict to gradient observables (these are observables which 
depend only on field increments) 
in order to have a well-defined infinite-volume measure. Viewed on the resulting incremental 
field variables  there is no monotonicity, even if the initial model in the finite-volume Gibbsian 
formulation has monotonicity.  A concrete example where this is very explicit is the Gaussian free field. 
Gradient Gibbs measures are beyond the framework of the present note, 
but we refer to~\cite{FuSp97,BiKo07,EnKu08,CoKu12,CoKu15}.

In another (simpler) application, the variable $\xi$ may also encode a parameter such as a random inverse temperature, modifying a given (random or non-random)
potential 
in a multiplicative way.  
More generally, we may consider any parametrized probability distribution on Gibbsian potentials with the same state space. Also in such a case it is clear that at different realizations, the structure of Gibbs measures may be completely different for different parameter values. 

If there is a smooth dependence on this parameter, it may be expected that the extremal states should depend smoothly on the parameter, at least in some regions of the phase diagram. This is known to be true for example for ordered states of ferromagnetic low-temperature models, in the regime where Pirogov-Sinai is applicable. For more complicated disordered systems, the possibility for such a  "state-following" may no longer be true.  A phenomenon of temperature chaos as it is 
called by the physicists  may appear, see~\cite{EnRu07,ChHo10,WaMaKa15,CoRi15,BeSuZe18}. We show that, while such chaos may happen, at least measurability is not lost.

\medskip
For a given non-random specification $\g$, we call a measure $\mu\in \PP(\O)$ an \emph{infinite-volume Gibbs measure} associated to $\g$, if it satisfies the DLR equations, i.e., if for all $\L\Subset S$ and $A\in \FF$ we have that $\int\mu(d\o)\g_\L(A|\o)=\mu(A)$. 
We denote by $\GG(\g)$ the set of infinite-volume Gibbs measures associated to $\g$. In the random case, we will assume that $\Gxi\neq\emptyset$ for $\P$-almost all $\xi$.

In particular, for $\P$-almost all $\xi$, $\Gxi$ is a non-empty simplex of probability measures in $\PP(\O)=\PP(\O,\FF)$ (the set of probability measures on $(\O,\FF)$).  $\Gxi$ can be represented by its extreme elements, which we denote by $\Gexxi$. As a result of the non-random Gibbsian theory, applied at fixed $\xi$, any $\mu\in \Gxi$ is the barycenter of a unique probability measure $w[\xi]_\mu$ on $\Gexxi$, see~\cite[Chapter 7]{Ge11}. The interpretation here is that elements of the extreme boundary represent \emph{pure} macroscopic states of the system, where the macroscopic variables are not random. Therefore, the weight measure $w[\xi]_\mu$, which appears in this decomposition of $\mu$, then represents the uncertainty encountered by experimentalists about which pure states have been observed.

\medskip
We will be concerned with \emph{random infinite-volume Gibbs measures} given by the \emph{measurable} map $\xi\mapsto \mu[\xi]$. Measurability w.r.t.~the parameter $\xi$ is very important, here it is defined w.r.t.~the \emph{evaluation sigma algebra} ${\it e}(\PP(\O))$ given by the smallest sigma algebra such that for all $A\in\FF$, the mappings
$$\xi\mapsto \mu[\xi](A)$$
are $\HH - \mathcal B[0,1]$ measurable.
For fixed $\xi$, the weight distribution $w[\xi]_{\mu[\xi]}$, associated to $\mu[\xi]$, is an element of $\PP(\PP(\O))=\PP(\PP(\O), {\it e}(\PP(\O)))$, the space of probability measures on the space of random fields $\PP(\O)$ equipped with the evaluation sigma algebra on $\PP(\O)$. In order to consider measurability of $\xi\mapsto w[\xi]_{\mu[\xi]}$, we need to equip also $\PP(\PP(\O))$ with a sigma algebra, which is the evaluation sigma algebra ${\it e}(\PP(\PP(\O)))$ where $\HH - \mathcal B[0,1]$ measurability of 
$$\xi\mapsto w[\xi]_{\mu[\xi]}(M)$$
has to be checked for all $M\in  {\it e}(\PP(\O))$.

\medskip
Our first main result now states the existence and measurability of the function $\xi\mapsto w[\xi]_{\mu[\xi]}$.

\begin{thm}\label{theorem1}Assume $(E,\EE)$ to be standard Borel. Let $\g[\xi]$ be a random specification on $\O=E^S$ such that $\GG(\g[\xi])\neq \emptyset$ 
for $\P$-almost all $\xi$ and let $\mu[\xi]$ be a random infinite-volume Gibbs measure in $\Gxi$ which 
exists $\P$-almost surely. 
Then, also the mapping $\xi\mapsto w[\xi]_{\mu[\xi]}$ exists and is $\HH - {\it e}(\PP(\PP(\O))) $ measurable. Moreover, 
\begin{enumerate} 
 \item the weight function is almost surely supported on the quenched extremal Gibbs measures, that is 
$w[\xi]_{\mu[\xi]} \in \PP(\Gexxi, {\it e}( \Gexxi))$ for $\P$-almost every $\xi \in H$, and
\item the random Gibbs measure has the measurable extremal decomposition 
\begin{equation}\label{Decomposition}
\mu[\xi]=\int_{\Gexxi} w[\xi]_{\mu[\xi]}(d\nu)\nu .
\end{equation}
\end{enumerate}
The map $\xi\mapsto w[\xi]_{\mu[\xi]}$ is defined by a $\xi$-measurable kernel $\pi[\xi]$ 
via 
$w[\xi]_{\mu}(M)=\mu(\omega:\pi[\xi](\cdot|\o)\in M)$ for all $M\in {\it e}(\PP(\O))$ and $\mu\in\PP(\O)$. 
\end{thm}
Considering the proof of Theorem~\ref{theorem1}, in particular the input given by Proposition~\ref{prop1}, we actually establish measurability, w.r.t.~$\xi$, of the affine mapping $w[\xi]: \mu\mapsto w[\xi]_\mu$ from $\PP(\O)$ to $\PP(\PP(\O), {\it e}( \PP(\O)))$ equipped with the evaluation sigma algebra $\s\big( {\it e}_{\mu,M}: (\mu, M)\mapsto w[\xi]_\mu(M), \mu\in\PP(\O), M\in  {\it e}( \PP(\O))\big)$. Additionally, under the assumptions of the theorem, for $\P$-almost all $\xi$, $\mu\mapsto w[\xi]_\mu$ is a bijection between $\Gxi\subset\PP(\O)$ and $\PP\big(\Gexxi, {\it e}(\Gexxi)\big)$.
%

\medskip
What have we gained? Ideally one would like to have a measurable extreme decomposition $\xi\mapsto \mu^{\alpha}[\xi]$ taking values in the extremals for given $\xi$, of the same type, indexed by some index $\a$. This can be shown to be true for some systems, e.g., for the random-field Ising model, where $\a\in\{+,-\}$. Indeed, by the FKG inequality, one may construct a maximal measure $\mu^+[\xi]$ with a fixed $+$-boundary condition, for any realization of $\xi$. However, in general this will not be possible, and so the best result we can hope to obtain is to go one level higher and construct a measurable map from $\xi$ to the {\em probability measures} on the \emph{extremals} for the same $\xi$. This is precisely the result of Theorem~\ref{theorem1}.

\subsection{Metastates}\label{Metastates}
Let us start this section by giving some background on metastates. In~\cite{AiWe90}, Aizenman and Wehr constructed a metastate as a $\xi$-dependent probability measure on the Gibbs states of a system that describes the large-volume asymptotics. The main underlying ideas for this are the following. Consider the probability distribution of the pair of the finite-volume Gibbs measure and the disorder variable 
\begin{align}\label{Pair}
(\g_{\L_n}[\xi]( \cdot |\o ),\xi)
\end{align} 
under the governing measure of the disorder variable, $\P(d\xi)$, for some fixed boundary condition $\o$. Let us note that it is useful to consider also more generally open or periodic boundary conditions, instead of fixed configurations $\o$. Suppose now that a limit exists for the random pair~\eqref{Pair} in the sense of weak convergence, along the cofinal sequence $(\L_n)_{n\in\N}$. Call the resulting limiting distribution $K(d \mu, d \xi)$. Of this limit, we may take a conditional distribution, obtained by conditioning on the disorder variable $\xi$, and this provides us with a measure on the first variable, which we call $\k^\text{AW}[\xi](d\mu)$, the {\em Aizenman-Wehr} or {\em conditional metastate}. 
The existence of the limit $K$ can be ensured for volume-subsequences $(\L_{k_n})_{n\in\N}$, for compact local state spaces $E$, and nice specifications, although we are not aware that this result appears in full generality. To carry this out in detail however is not a problem for example for the aforementioned random Ising models using the finite local state space and Markovianity of the interaction. The independence of the limit on the choice of the subsequence has not been proved in general, but it is very plausible in all examples that have been studied.

There is also a relation to the \emph{Newman-Stein metastate}, see~\cite{NeSt98}, 
which is constructed as a limit point of 
\begin{equation}\label{Empi}
\begin{split}
\k_{N}[\xi]=\frac{1}{N}\sum_{n=1}^ N\d_{\g_{\L_n}[\xi]( \cdot |\o )}
\end{split}
\end{equation}
for sufficiently sparse $\L_n$. 
This provides the interesting and intuitive reinterpretation of the Aizenman-Wehr-metastate as giving a distribution on the 
Gibbs states at fixed $\xi$ by ``drawing a large volume at random, independently 
of the choice of $\xi$".
Covering both approaches, let us define \emph{metastates} as follows. 
\begin{defn} \label{metastatedefinition} A measurable map $\k:H\rightarrow \PP(\PP(\O))$ is called a \emph{metastate} if 
for $\P$-almost all $\xi$, the Gibbs measures corresponding to $\xi$ obtain full mass, i.e., 
$$\k[\xi]\big(\GG(\g[\xi])\big)=1.$$
\end{defn}

\medskip
Coming back to our results, in Theorem~\ref{theorem1}, we assumed the existence of a measurable Gibbs state $\mu[\xi]$. 
To see how this assumption can be satisfied, note that for a metastate 
$\k[\xi](d\mu)$ one may take its expected value (or barycenter) 
 $\int \k[\xi](d\nu) \nu=\mu_{\k}[\xi]$ to obtain a measurable map into the (not necessarily 
 extreme) Gibbs measures for $\xi$. Further, in the sense of Definition~\ref{metastatedefinition}, the measurable weight function 
$\xi\mapsto w[\xi]_{\mu[\xi]}$ is a metastate. The 
statement of Theorem~\ref{theorem1} then has the following rephrasing: 
For any measurable Gibbs measure $\mu[\xi]$ there is a metastate $w[\xi]_{\mu[\xi]}$ 
supported on the quenched {\it pure} states $\Gexxi$ such that measurable Gibbs measure 
can be written as the barycenter of this metastate, in the form $\int_{\PP(\O)}  w[\xi]_{\mu[\xi]}(d\nu)\nu$.

In particular, applying Theorem~\ref{theorem1} and using the intermediate step to define $\mu[\xi]=\mu_{\k}[\xi]$, we obtain the 
following result.  
\begin{cor}\label{theorem3} Suppose $\k:H\rightarrow \PP(\PP(\O))$ is a metastate. 
Then there is an associated decomposition metastate $\k^{\rm ex}:H\rightarrow \PP(\PP(\O))$ which has the following two properties. 
\begin{enumerate}
\item $\k^{\rm ex}$ has the same barycenter as $\k$, i.e., for  $\P$-almost all $\xi$,
$$
\int_{\PP(\O)}\k^{\rm ex}[\xi](d\nu) \nu =\int_{\PP(\O)} \k[\xi](d\nu) \nu \qquad\text{ and}
$$ 
\item
for  $\P$-almost all $\xi$, $\k^{\rm ex}$ is supported on the corresponding extremal Gibbs measures, 
i.e., $$\k^{\rm ex}[\xi]\big(\Gexxi\big)=1.$$
\end{enumerate}
\end{cor}

\subsection{Symmetries}
So far we have not used any specific structure for the space $(H,\HH)$. For the next result we want to introduce translations on $H$ as well as on $\O$. For this, let us assume that $S=\Z^d$ and equip $H$ with a \emph{shift group} $\theta=(\theta_i)_{i\in S}$, i.e., a family of measurable bijections $\theta_i: H\to H, \, \xi\mapsto \theta_i\xi$. More specifically, for $\o\in\O$, we define $(\theta_j\o)_i=\o_{i-j}$. For example in the random-field Ising model, where $H=\{-1,1\}^S$, the shift group is defined exactly as for $\O$. We call a random specification $\g$ \emph{translation covariant}, if for all $j\in S, \o\in\O$, $A\in\FF$ and $\L\Subset S$ we have that 
\begin{equation*}
\g_{\theta_j\L}[\theta_j\xi](\theta_jA|\theta_j\o)=\g_\L[\xi](A|\o)\qquad\text{for }\P\text{-almost all }\xi
\end{equation*}
where $\theta_j A=\{\theta_j\o: \o\in A\}$ and $\theta_j\L=\{i+j\in S: i\in\L\}$. Similarly we call a random Gibbs measure $\xi\mapsto\mu[\xi]$ \emph{translation covariant}, if for all $A\in\FF$ and $j\in S$ we have that 
\begin{equation*}
\mu[\theta_j\xi](A)=\mu[\xi](\theta_{-j}A)\qquad\text{for }\P\text{-almost all }\xi.
\end{equation*}
The following theorem expresses that the extremal decomposition of Theorem~\ref{theorem1} for translation invariant models is also translation covariant. 
\begin{thm} \label{theorem2} Assume that $\g[\xi]$ and $\mu[\xi]$ satisfy the conditions of Theorem~\ref{theorem1} and are both translation covariant.
Then, 
the measurable map $\xi \mapsto w[\xi]_{\mu[\xi]}$ from Theorem~\ref{theorem1}
is also translation invariant, i.e., for all $j\in S$ and $M\in \it{e}(\PP(\O))$ we have that
\begin{equation}\label{cov1}
w[\theta_j\xi]_{\mu[\theta_j\xi]}(\theta_jM)= w[\xi]_{\mu[\xi]}(M)\qquad\text{for }\P\text{-almost all }\xi,
\end{equation}
where $\theta_jM=\{\theta_j\nu\in\PP(\O): \nu\in M \}$ with $\theta_j\nu(A)=\nu(\theta_{-j}A)$ for all $A\in \FF$.
\end{thm}
Under suitable conditions, Theorem~\ref{theorem2} can be extended also to cover other group actions $\tau$ instead of $\theta$, such as spin flips or spin rotations, generalizing~\cite[Corollary 7.28]{Ge11} to the $\xi$-dependent setting.

\section{Proofs}
The proof of Theorem~\ref{theorem1} is based on the existence of the $(\GG(\g[\xi]),\TT)$-kernel
 $\pi[\xi](\cdot |\o)$ with the desired measurability and support properties. Once this is established, the proof of the main theorem follows by an application of~\cite[Proposition 7.22]{Ge11} and measure-theoretic arguments. 

\medskip
Let us introduce some more notation. For $\L\Subset S$ we denote by $\FF_\L$, the sub-sigma algebra of sets in $\FF$ which are measurable in $\L$. Further, the \emph{tail sigma algebra} $\TT$ is defined via $\TT=\bigcap_{\L\Subset S}\FF_{\L^c}$ where $\L^c=\O\setminus\L$. As a particular case of ~\cite[Definition 7.21]{Ge11}, a probability kernel $\pi$ is called a \emph{$(\GG(\g),\TT)$-kernel}, for $\g$ a specification, if the following holds. For all measures $\mu\in\GG(\g)$,
\begin{enumerate}
\item $\pi$ is a version of the conditional probability given $\TT$, that is, for all $A\in\FF$, $\mu(A |\TT)=\pi(A|\cdot)$ $\mu$-almost surely, and
\item for $\mu$-almost every $\o$, the measure $\pi(\cdot | \o)$ is an element in $\GG(\g)$ and $\{\o: \pi(\cdot|\o)\in\GG(\g)\}\in \TT$.
\end{enumerate}
Let us stress that $\pi$, if it exists, does not depend on individual measures $\mu$, but rather on $\g$. 

\medskip
The following result establishes existence of a random version of a $(\GG(\g),\TT)$-kernel and therefore extends~\cite[Proposition 7.25]{Ge11} to the random setting.
\begin{prop}\label{prop1}Assume $(E,\EE)$ to be standard Borel. Let $\g[\x]$ be a random specification on $\O=E^S$ such that $\Gxi\neq\emptyset$ for $\P$-almost all $\xi$. Then there exists a map $\pi:H\times \O \times \FF \mapsto [0,1]$,  
 $(\xi,\o,A)\mapsto \pi[\xi](A|\o)$, such that 
\begin{enumerate}  
\item at any fixed $A\in \FF$, we have $\bigcap_{\L}(\HH\otimes \FF_{\L^c}) - \mathcal B[0,1]$ measurability of 
$$(\xi,\o)\mapsto\pi[\xi](A|\o),\qquad\text{ and}$$
\item for $\P$-almost all $\xi$, $\pi[\xi]$ is a $(\GG(\g[\xi]), \TT)$-kernel.
\end{enumerate} 
\end{prop}
We will use the following measure-theoretic generalities for the proofs. 
Since $(E,\EE)$ is assumed to be standard Borel, also $(\O,\FF)$ is a standard Borel space. Moreover, if the configuration space is standard Borel, i.e., measure-theoretic isomorphic to a complete separable metric space with Borel sigma-algebra, then, $\FF$ can be generated by a countable core. We recall that a core $\C$ is a finite-intersection stable 
generator of the sigma algebra $\FF$ with the following property. Convergence of $(\mu_n(A))_{n\in\N}$ for all events $A\in \C$, where $(\mu_n)_{n\in\N}\subset\PP(\O)$, uniquely determines a limiting measure $\mu\in\PP(\O)$ with $\mu(A)=\lim_{n\uparrow\infty}\mu_n(A)$ for all $A\in\C$, see~\cite[Definition 4.A9]{Ge11}.

\medskip
Before we prove Proposition~\ref{prop1}, let us show how it can be used to prove Theorem~\ref{theorem1}. 
\begin{proof}[Proof of Theorem~\ref{theorem1}]
By hypothesis, for $\P$-almost all $\xi$, we have that $\Gxi$ is non-empty. By Proposition~\ref{prop1}, the measurable kernel $\pi$ exists and, in particular, by part 2 of Proposition~\ref{prop1} is a $(\GG(\g[\xi]), \TT)$-kernel. Thus, for $\P$-almost all $\xi$, using~\cite[Proposition 7.22]{Ge11}, we have the following. For every $\mu\in \Gxi$, there exists a unique $w[\xi]_{\mu}$, supported on $\Gexxi$, such that the extremal decomposition~\eqref{Decomposition} holds. In particular, the mapping $\xi\mapsto w[\xi]_{\mu[\xi]}$ exists and $w[\xi]_{\mu[\xi]}(M)=\int\mu[\xi](d\o)\one\{\pi[\xi](\cdot|\o)\in M\}$ for all $M\in {\it e}(\PP(\O))$. 
Hence, all that remains to be shown is that the mapping $\xi\mapsto w[\xi]_{\mu[\xi]}$ is $\HH - {\it e}(\PP(\PP(\O)))$ measurable.

\medskip
By the definition of the evaluation sigma algebra, it suffices to show $\HH - \mathcal B[0,1]$ measurability of 
$$\xi\mapsto (\xi,\mu[\xi])\mapsto w[\xi]_{\mu[\xi]}(M)$$ 
for all $M\in {\it e}(\PP(\O))$. For this, note that the first mapping is measurable by assumption since we supposed that $\xi\mapsto\mu[\xi]$ is $\HH -{\it e}(\PP(\O))$ measurable. 
Thus, it suffices to establish $\HH\otimes{\it e}(\PP(\O))-\mathcal B[0,1]$ measurability of 
\begin{equation}\label{Map1}
(\xi,\mu)\mapsto \int\mu(d\o)\one\{\pi[\xi](\cdot|\o)\in M\}.
\end{equation}
By the definition of ${\it e}(\PP(\O))$ we can limit ourself to check sets $M$ of the form $M(A,c)=\{\nu\in \PP(\O): \nu(A)\leq c\}$ with $A\in \FF$ and $c\in[0,1]$ in which case~\eqref{Map1} becomes
\begin{equation*}
(\xi,\mu)\mapsto \int\mu(d\o)\one\{\pi[\xi](A|\o)\le c\}.
\end{equation*}
Here, by Proposition~\ref{prop1} part 1 we have that $\{(\xi,\o): \pi[\xi](A|\o)\le c\}\in\bigcap_{\L}(\HH\otimes \FF_{\L^c})\subset (\HH\otimes \FF)$ for all $A\in\FF$ and $c\in [0,1]$.
In fact, as is a standard statement in the context of Fubini's theorem, for any $\HH\otimes \FF - \mathcal B[0,1]$ measurable bounded $g$ we have that  
$$(\xi,\mu)\mapsto \int \mu(d\o)g(\xi,\o)
$$ 
is $\HH \otimes  \text{{\it e}}(\PP(\O))-\mathcal B[0,1]$ measurable. Indeed, for any $C\in \HH\otimes \FF$ we have that 
$$(\xi,\mu)\mapsto \int \mu(d\o)\one\{(\xi,\o)\in C\}$$ 
is $\HH \otimes  \text{{\it e}}(\PP(\O)) - \mathcal B[0,1]$ measurable, since 
$$\DD=\{ C \in \HH\otimes \FF: (\xi,\mu)\mapsto \int \mu(d\o)\one\{(\xi,\o)\in C\}\text{ is }\HH \otimes  \text{{\it e}}(\PP(\O))-\mathcal B[0,1] \text{ measurable}\}
$$
contains the system of squares $Q=\{A\times A': A\in \HH, A' \in \F\}$ and is a Dynkin system.  
The statement that $Q\subset\DD$ follows, since 
$$(\xi,\mu)\mapsto \int \mu(d\o)\one\{(\xi,\o)\in A \times A'\}= \mu(A)\one\{\xi\in A'\}
$$
is measurable as a product of measurable functions. The fact that it is a Dynkin system 
is obvious since for $C_1 \supset C_2$ 
$$\int \mu(d\o)\one\{(\xi,\o)\in C_1 \ba C_2\}
= \int \mu(d\o)\one\{(\xi,\o)\in C_1\}- \int \mu(d\o)\one\{(\xi,\o)\in C_2\}
$$
where the l.h.s.~is measurable if the two terms on the r.h.s.~are, and 
a similar argument applies for countable disjoint unions. Thus, by Dynkin's theorem~\cite[Theorem 1.19]{Kl08}, the sigma-algebra generated by 
an intersection-stable system of sets equals the Dynkin system generated by it. 
We have $Q\subset\DD$ and hence $\HH\otimes\FF=\s(Q)=\DD(Q)\subset \DD$, and so $\HH\otimes\FF=\DD$ follows.
This finishes the proof.
\end{proof}
\begin{proof}[Proof of Proposition~\ref{prop1}]
Since $(E,\EE)$ is standard Borel, as outlined before the proof on Theorem~\ref{theorem1}, there exists a countable core of $\FF$ which we denote $\C$, for details see~\cite[Theorem 4.A11]{Ge11}. Only at this moment we will make use of this notion, namely in the definition of the following set. 
Extending the construction in~\cite[Proposition 7.25]{Ge11} we define
\begin{equation*}
\bar \O_o=\{ (\x,\o)\in H \times \O: \lim_{n\uparrow \infty} \g_{\L_n}[\xi](A | \o) {\text{ exists for all }}A\in \C\},
\end{equation*}
the set of disorder parameters and configurations such that the specification has an infinite-volume  limit on the core. Here $(\L_n)_{n\in\N}$, denotes a cofinal sequence of finite subsets of $S$. While for some $(\xi,\o)$ the convergence property might depend on the choice of the sequence, this will not matter in the construction, as we will see. In particular, $\bar \O_o=\bigcap_{A\in\C}\bar\O_o(A)$ where $\bar\O_o(A)=\{ (\x,\o)\in H \times \O : 
\lim_{n\uparrow \infty} \g_{\L_n}[\xi](A | \o) {\text{ exists}}\}$ since $\C$ is assumed to be countable. 
Further, $\bar\O_o(A)\in \bigcap_{\L}(\HH\otimes \FF_{\L^c})$ for all $A\in \C$
and thus $\bar \O_o  \in \bigcap_{\L}(\HH\otimes \FF_{\L^c})$. 
Define $ \O_o(\xi)=\{\o\in\O : (\x,\o)\in \bar \O_o\}$ and note that $ \O_o(\xi)\in \TT$ lies 
in the tail sigma algebra corresponding to $\FF$. 

\medskip
Fixing a disorder parameter $\xi$, we can repeat the arguments from~\cite[Proposition 7.25]{Ge11} to define a kernel $\pi[\xi]$ and prove part 2 of the Proposition. For this also fix $\o\in \O_o(\xi)$. 
Then, by the definition of the core $\C$, there is a unique probability measure $\pi[\xi](\cdot | \o)\in \PP(\O)$ such that $\pi[\xi](A | \o)=\lim_{n\uparrow\infty} \g_{\L_n}[\xi](A | \o)$ for all $A \in \C$. 
We extend $\pi[\xi](\cdot |\o)=\nu_o$ for an arbitrary measure $\nu_o\in \PP(\O)$ whenever $(\xi,\o)$ are not in the set $\bar \O_o$. Since, for $\P$-almost all $\xi$ we have that $\Gxi\neq\emptyset$, repeating the exact same arguments as in~\cite[Proposition 7.25]{Ge11}, we see that for $\P$-almost all $\xi$, $\pi[\xi]$ is indeed a $(\GG(\g[\xi]), \TT)$-kernel.

\medskip
The preceding arguments established existence of the map $\pi$. What remains to be shown is part 1. We must show that $(\xi,\o)\mapsto \pi[\xi](A |\o)$ is $\bigcap_{\L}(\HH\otimes \FF_{\L^c}) - \mathcal B[0,1]$ measurable 
at any $A \in \FF$. For this define 
$$\DD=\{ A \in \FF: (\xi,\o)\mapsto \pi[\xi](A |\o) \text{ is } \bigcap_{\L}(\HH\otimes \FF_{\L^c}) - \mathcal B[0,1]\text{ measurable}\},$$
and recall that for any $A\in\C$,
\begin{equation}
\pi[\xi](A | \o)=\one\{(\xi,\o)\in \bar \O_o\}\lim_{n\uparrow\infty}\g_{\L_n}[\xi](A | \o)+\one\{(\xi,\o)\not\in \bar \O_o\}\nu_o(A).
\end{equation}
Thus, $\C$ is contained in $\DD$, since for all $A\in \C$ the limit $\lim_{n\uparrow\infty} \g_{\L_n}[\xi](A | \o)$ as well as the indicators $\one\{\bar \O_o\}$ and $\nu_o(A)\one\{\bar \O_o\}$ are $\bigcap_{\L}(\HH\otimes \FF_{\L^c}) - \mathcal B[0,1]$ measurable function
and hence $(\xi,\o)\mapsto \pi[\xi](A | \o)$ is a $\bigcap_{\L}(\HH\otimes \FF_{\L^c}) - \mathcal B[0,1]$ measurable function. 

Note also that $\DD$ is a Dynkin system. For, let $A\subset B$ be events in $\DD$, then 
\begin{equation}\begin{split}
\pi&[\xi](B\backslash A | \o)=\one\{(\xi,\o)\in \bar \O_o\}\lim_{n\uparrow\infty}\g_{\L_n}[\xi](B\backslash A  | \o)+\one\{(\xi,\o)\not\in \bar \O_o\}\nu_o(B\backslash A )\cr
&=\one\{(\xi,\o)\in \bar \O_o\}\big(\lim_{n\uparrow\infty}\g_{\L_n}[\xi](B| \o) - \lim_{n\uparrow\infty}\g_{\L_n}[\xi](A| \o)\big) 
+\one\{(\xi,\o)\not\in \bar \O_o\}\big(\nu_o(B)-\nu_o(A)\big) \cr
&=\pi[\xi](B| \o)- \pi[\xi](A| \o).
\end{split}
\end{equation}
Hence measurability of the r.h.s.~carries over to that of the l.h.s.. 
In the same way it is seen that $\pi[\xi](\bigcup_n A_n | \o)=\sum_n \pi[\xi](A_n| \o)$ 
is measurable for disjoint $A_n$ when each of the $ \pi[\xi](A_n| \o)$  is. 
Thus, by Dynkin's theorem we have that $\C\subset\DD$ and hence $\FF=\s(\C)=\DD(\C)\subset \DD$, and so $\FF=\DD$ follows, this completes the proof.  
\end{proof}

\begin{proof}[Proof of Theorem~\ref{theorem2}]
We need to go back to the definition of the metastate as an image of the $\xi$-measurable kernel under 
the random Gibbs measure, and check how these objects behave under translation. Recall that
\begin{equation*}
w[\xi]_{\mu[\xi]}(M)=\int\mu[\xi](d\o)\one\{\pi[\xi](\cdot|\o)\in M\},
\end{equation*}
and thus for $\P$-almost all $\xi$, by a change of variables,
\begin{equation*}
\begin{split}
w[\theta_j\xi]_{\mu[\theta_j\xi]}(\theta_jM)&=\int\mu[\theta_j\xi](d\o)\one\{\pi[\theta_j\xi](\cdot|\o)\in \theta_jM\} \\
&=\int\mu[\xi](d(\theta_{-j}\o))\one\{\pi[\theta_j\xi](\cdot|\o)\in \theta_jM\} \\
&=\int\mu[\xi](d\o)\one\{\pi[\theta_j\xi](\cdot|\theta_j\o)\in \theta_jM\},
\end{split}
\end{equation*}
where the second step follows from the translation covariance of $\mu[\xi]$. Now, in order to conclude equation~\eqref{cov1}, it suffices to show that $\pi[\theta_j\xi](\cdot|\theta_j\o)\in \theta_jM$ if and only if $\pi[\xi](\cdot|\o)\in M$. For this, first note that the statement that there exists $\nu\in\theta_jM$ such that for all $A\in\FF$ we have 
$$\pi[\theta_j\xi](\theta_jA|\theta_j\o)=\nu(\theta_jA),$$
since $\nu(\theta_jA)=\theta_{-j}\nu(A)$, is equivalent to the statement that there exists $\nu\in M$ such that for all $A\in\FF$ we have 
$$\pi[\theta_j\xi](\theta_jA|\theta_j\o)=\nu(A),$$
since $\theta_{-j}\nu\in M$. 
Thus, it suffices to show that for all $j\in S$, $A\in \FF$ and $\o\in \O$ we have that 
\begin{equation*}
\pi[\theta_j\xi](\theta_jA|\theta_j\o)=\pi[\xi](A|\o)\qquad\text{for }\P\text{-almost all }\xi.
\end{equation*}
We modify slightly previous arguments to reflect the translation-invariant setting. First, note that for a product spin space of the form $\O=E^S$ with $S=\Z^d$ and with a standard Borel set $E$, 
the countable core $\C$ can be chosen in a translation-invariant way, meaning that for each subset of $\C$ 
also its translate belongs to $\C$, see~\cite[Corollary 4.A.13]{Ge11}, and the discussion preceding it. We quickly recall the arguments: $E$ is measure-theoretic isomorphic either to $\{1, \dots, |E| \}$ with the product set-sigma algebra, if $|E|<\infty$; 
$\N$ with the product set-sigma algebra, if $|E|$ is countably infinite; or $\{0,1\}^{\N}$ with sigma-algebra generated by the cylinder sets, if $E$ is uncountably infinite. 
For instance, in the uncountably infinite case $\O$ is isomorphic to $\{0,1\}^{\N\times \Z^d}$, and the countable core can be chosen as the set of finite cylinders of this set. The set of cylinders is translation invariant, in particular under translations in $\Z^d$. 
Second, we use a slight modification of the kernel 
\begin{equation*}
\pi[\xi](A | \o)=\one\{(\o,\xi)\in \bar \O_o\}\lim_{n\uparrow\infty}\g_{\L_n}[\xi](A | \o)+\one\{(\o,\xi)\not\in \bar \O_o\}\nu_o(A),
\end{equation*}
defined now via the set 
\begin{equation*}
\label{tag}
\bar \O_0=\{ (\x,\o)\in H \times \O : 
\lim_{n\uparrow \infty} \g_{\theta_i\L_n}[\xi](A | \o) {\text{ exists for all }}A\in \C\text{ and all }i\in\Z^d,\text{ and is independent of }i
\},
\end{equation*}
where convergence is now guaranteed in a translation invariant way.
Using this, we have that  
\begin{equation*}\label{gu}
\begin{split}
\pi[\theta_j\xi](\theta_jA |\theta_j\o)&=\one\{(\theta_j\o,\theta_j\xi)\in \bar \O_o\}\lim_{n\uparrow\infty}\g_{\theta_j\L_n}[\theta_j\xi](\theta_jA |\theta_j \o)+\one\{(\theta_j\o,\theta_j\xi)\not\in \bar \O_o\}\nu_o(\theta_jA)\cr
&=\one\{(\o,\xi)\in \bar \O_o\}\lim_{n\uparrow\infty}\g_{\L_n}[\xi](A |\o)+\one\{(\o,\xi)\not\in \bar \O_o\}\nu_o(A)
=\pi[\xi](A |\o),
\end{split}
\end{equation*}
when we choose $\nu_o$ in a translation-invariant way. 
This concludes the proof.
\end{proof}

\bibliography{Jahnel}
\bibliographystyle{alpha}

\end{document}